\documentclass[10pt]{conm_p_l}
\usepackage{amssymb,latexsym, amscd}
\usepackage[square, numbers]{natbib}
\usepackage[all]{xy}
\usepackage{graphicx}
\usepackage{mathrsfs}
\usepackage{amsmath}

 \newtheorem{theorem}{Theorem}[section]

 \newtheorem{lemma}[theorem]{Lemma}
 \newtheorem{proposition}[theorem]{Proposition} \theoremstyle{definition}
 \newtheorem{definition}[theorem]{Definition}
 \theoremstyle{definition}
 
 \theoremstyle{remark}
 \newtheorem{rem}[theorem]{Remark}
 \numberwithin{equation}{section}

\newcommand{\ben}{\begin{equation}}
\newcommand{\een}{\end{equation}}

\newcounter{commentcounter}

\newcommand{\integer}{\ensuremath{{\mathbb Z}}}

\newcommand{\real}{\ensuremath{{\mathbb R}}}
\newcommand{\complex}{\ensuremath{{\mathbb C}}}

\newcommand{\LL}{\mathcal{L}}

\newcommand{\gr}{\mathfrak}

\begin{document}

\title[Equivariant extensions of differential forms]{Equivariant extensions of differential forms for non-compact Lie groups}

\author[H. Garc\'{\i}a-Compe\'an]{Hugo Garc\'{\i}a-Compe\'an}
\thanks{The third author acknowledges and thanks the financial support of the Alexander Von Humboldt Foundation.}
\dedicatory{In celebration of the fiftieth anniversary of the
Mathematics Department at the CINVESTAV.}
\address{Departamento de F\'{\i}sica, Centro de Investigaci\'on y de Estudios Avanzados,
Av. I.P.N 2508, Zacatenco, Delegaci\'on Gustavo A. Madero, C.P. 07360,
, M\'exico D.F., M\'exico} \email{compean@fis.cinvestav.mx}
\author[P. Paniagua]{Pablo Paniagua}
\address{Departamento de Matem\'aticas, Escuela Superior de F\'{\i}sica y Matem\'aticas del  Instituto Polit\'ecnico Nacional,
Unidad Adolfo L\'opez Mateos, Edificio 9, 07738, M\'exico D.F.,
M\'exico }
 \email{ppaniagua@esfm.ipn.mx}
\author[B. Uribe]{Bernardo Uribe }
\address{ Departamento de Matem\'{a}ticas, Universidad de los Andes,
Carrera 1 N. 18A - 10, Bogot\'a, COLOMBIA}
 \email{buribe@uniandes.edu.co}


\begin{abstract}
Consider a manifold endowed with the action of a Lie group. We study
the relation between the cohomology of the Cartan complex and the
equivariant cohomology by using the equivariant De Rham complex
developed by Getzler, and we show that the cohomology of the Cartan
complex lies on the $0-th$ row of the second page of a spectral
sequence converging to the equivariant cohomology. We use this
result to generalize a result of Witten on the equivalence of
absence of anomalies in gauged WZW actions on compact Lie groups to
the existence of equivariant extension of the WZW term, to the case
on which the gauge group is the special linear group with real
coefficients.
\end{abstract}

\maketitle

\section{Introduction}

In certain situations, geometrical information of manifolds might be
encoded in differential forms. In the presence of symmetries of the
manifold via the action of a Lie group, the behavior of these
differential forms under the group action may lead to a better
understanding of the manifold itself. In some well known instances
of actions of compact Lie groups, the action is of a particular type
whenever the differential form may be extended to an equivariant one
in the Cartan model of equivariant cohomology \cite{Cartan}; this is
the case for example in Hamiltonian actions on symplectic manifolds
\cite{AtiyahBott}, in Hamiltonian actions on exact Courant
algebroids \cite{Bursztyn, CaviedesHuUribe, Hu2, Uribe-dg} or in
gauged WZW actions which are anomaly free \cite{Witten} and its
generalizations \cite{GarciaCompeanPaniagua, Paniagua}.

When the Lie group is not compact, the Cartan complex associated to
the action in general does not compute the equivariant cohomology of
the manifold, and therefore many of the results that hold for
compact Lie groups may not hold for the non-compact case. But since
the Cartan model is very well suited for studying the infinitesimal
behavior of the differential forms with respect to the action of the
Lie algebra, it would be worthwhile knowing more about the relation
between the cohomology of the Cartan complex and the equivariant
cohomology. In this paper we study this relation and we obtain some
interesting results which in particular permit us generalize the
conditions for cancellation of anomalies on gauge WZW actions
developed by Witten \cite{Witten}, to the case on which the gauge
group is the non-compact group $SL(n,\real)$.

The main ingredient of this work is the equivariant De Rham complex
developed by Getzler \cite{Getzler} whose cohomology calculates the
equivariant cohomology independent whether the group is compact or
not. We show that there is an inclusion of complexes of the Cartan
complex into the equivariant De Rham one, and therefore we obtain a
homomorphism between the cohomology of the Cartan complex to the
equivariant De Rham cohomology. This in particular implies that any
closed form in the Cartan complex defines an equivariant cohomology
class, but note that the converse may not be true. We furthermore
show that there is a spectral sequence converging to the equivariant
De Rham cohomology such that its $E_2^{*,0}$-term is isomorphic to
the cohomology of the Cartan complex, and in this way an explicit
relation between the cohomology of the Cartan complex and the
equivariant De Rham cohomology is obtained. Now, since the first
page of this spectral sequence could be understood in terms of the
differential cohomology of a Lie group, we recall its definition and
some of its properties and we reconstruct some calculations for the
group $SL(n,\real)$.

We conclude this work by applying our results to the gauged WZW
actions whenever the gauge group is $SL(n,\real)$. We recall the
result of Witten which claims that the gauged WZW action on compact
Lie groups is anomaly free, if and only if the WZW term $\omega =
\frac{1}{12 \pi} {\rm{Tr}}(g^{-1} dg)^3$ can be extended to a closed
form in the Cartan complex, and we generalize it to the case on
which the gauge group is $SL(n,\real)$; this is Theorem
\ref{equivalence of extension for SL(n,R)}. We use this theorem to
construct explicit examples where the equations representing the
condition of anomalies cancelation hold and where they do not. The
physical implications of our work will appear elsewhere.

\vspace{0.5cm}

\noindent {\bf Acknowledgments} The first author was partly supported 
 by the CONACyT-M\'exico through grant number 128761. The third author would
like to acknowledge and to thank the financial support provided by
the Alexander Von Humboldt Foundation and also would like to
 thank the hospitality of Prof. Wolfgang L\"uck at the Mathematics Institute of the
University of Bonn.

\section{Equivariant cohomology}

Let $G$ be a Lie group and $M$ a manifold on which $G$ acts on the
left by diffeomorphisms . The $G$-equivariant cohomology of $M$
could be defined as the singular cohomology of the homotopy quotient
$EG \times_G M$
$$H^*_G(M; \integer) : = H^*(EG \times_G M; \integer),$$
where $EG$ is the universal $G$-principal bundle $G \to EG \to
BG$.

The previous definition works for any topological group and any
continuous action, but sometimes it is convenient to have a De Rham
version with differentiable forms for the equivariant cohomology
whenever the group is of Lie type and the action is
differentiable.When the group $G$ is compact, the Weil and Cartan
models provide a framework on which the equivariant cohomology could
be obtained via a complex whose ingredients are the local action of
the Lie algebra $\gr{g}$, and the differentiable forms
$\Omega^\bullet M$. When the group $G$ is not compact there is a
more elaborate model for equivariant cohomology that we will
describe in the next section.

\subsection{De Rham model of Equivariant cohomology} \label{equivariant De Rham complex}
One way to obtain a De Rham model for equivariant cohomology is through the total complex
of the double complex
$$\Omega^*\left(N_\bullet ( G \ltimes M)\right)$$
that is obtained after applying the differentiable forms functor to
the simplicial space $N_\bullet ( G \ltimes M)$ which is the nerve
of the differentiable groupoid $G \ltimes M$.

By the works of Bott-Shulman-Stasheff \cite{BottShulmanStasheff} and
Getzler \cite{Getzler} we know that one way to calculate the
cohomology of the total complex of the double complex of differentiable
forms
 $$\Omega^*\left(N_\bullet ( G \ltimes M)\right)$$
 is through the differentiable cohomology groups
 $$H^*(G,S{\gr g}^* \otimes \Omega^\bullet M  )$$
of the group $G$ with values on the differentiable forms of $M$
tensor the symmetric algebra of the dual of the Lie algebra
$\gr{g}$.

In \cite{Getzler} Getzler has shown that there is a De Rham theorem
for equivariant cohomology showing that there is an isomorphism of
rings
$$H^*(G, S{\gr g}^* \otimes \Omega^\bullet M ) \cong H^*(EG \times_G M ; \real)$$
between the De Rham model for equivariant cohomology and the
cohomology of the homotopy quotient.

The De Rham model for equivariant cohomology defined by Getzler is
described as follows.

Consider the complex $C^k(G,  S{\gr g}^* \otimes \Omega^\bullet M
)$ with elements smooth maps
$$f(g_1, \dots , g_k | X) : G^k \times \gr{g} \to \Omega^\bullet M, $$
which vanish if any of the arguments $g_i$ equals the identity of
$G$. The operators $d$ and $\iota$ are defined by the formulas
\begin{eqnarray*}
 (df)(g_1, \dots , g_k | X) &=& (-1)^k df(g_1, \dots , g_k | X) \ \ \ \ \ \ {\rm{and}}\\
(\iota f) (g_1, \dots , g_k | X) &=& (-1)^k \iota(X) f(g_1, \dots
, g_k | X),
\end{eqnarray*} as in the case of the differential in Cartan's
model for equivariant cohomology \cite{Cartan, Guillemin}. Recall
that the elements in $\gr{g}^*$ are defined to have degree 2, and
therefore the operator $\iota$ has degree 1. Denote the generators of the symmetric algebra
$S\gr{g}^*$ by $\Omega^a$ where $a$ runs over a base of $\gr{g}$.

The coboundary $\bar{d}: C^k \to C^{k+1}$ is defined by the formula
\begin{eqnarray*}
(\bar{d}f)(g_0, \dots , g_k|X) & = & f( g_1, \dots , g_k | X ) +
 \sum_{i=1}^k (-1)^i f(g_0, \dots, g_{i-1}g_i, \dots  , g_k | X)\\
 & & +(-1)^{k+1} g_k f(g_0, \dots , g_{k-1} | {\rm{Ad}}(g_k^{-1})X),
\end{eqnarray*}
and the contraction $\bar{\iota} : C^k \to C^{k-1}$ is defined by
the formula
\begin{eqnarray*}
(\bar{\iota}f)(g_1, \dots , g_{k-1}|X) & = & \sum_{i=0}^{k-1} (-1)^i
\frac{\partial}{\partial t}  f(g_1, \dots, g_i, e^{tX_i}, g_{i+1}
\dots  , g_{k-1} | X),
\end{eqnarray*}
where $X_i= {\rm{Ad}}(g_{i+1} \dots g_{k-1})X$.

If the image of the map
$$ f: G^k \to  S \gr g^* \otimes \Omega^\bullet M $$
is a homogeneous polynomial of degree $l$, then the total degree
of the map $f$ equals $deg(f)=k+l$. It follows that the structural
maps $d, \iota, \bar{d}$ and $\bar{\iota}$ are degree 1 maps, and
the operator $$d_G = d + \iota +\bar{d} + \bar{\iota}$$ becomes a
degree 1 map that squares to zero.

\begin{definition}
The elements of the complex
$$\left( C^*(G,
S{\gr g}^* \otimes \Omega^\bullet M ) , d_G \right)$$ will be called {\it equivariant De Rham forms} and
its cohomology $$H^*(G, S{\gr g}^* \otimes \Omega^\bullet M )$$
will be called the equivariant De Rham  cohomology.
\end{definition}

In \cite{Getzler} it was shown that the complex $\left( C^*(G,
S{\gr g}^* \otimes \Omega^\bullet M ) , d_G \right)$ together with
the cup product
\begin{eqnarray*}
(a \cup b)(g_1, ... , g_{k+l}|X) = (-1)^{l(|a|-k)} \gamma
a(g_1,..., g_k| {\rm{Ad}}(\gamma^{ -1})X) b(g_{k+1},...,
g_{k+l}|X)
\end{eqnarray*}
for $\gamma = g_{k+1} ...g_{k+l}$, becomes a differential graded
algebra, and moreover that there is a canonical isomorphism of
rings
$$H^*(G, S{\gr g}^* \otimes \Omega^\bullet M ) \cong H^*(M \times_G
EG ; \real)$$ with the cohomology of the homotopy quotient.

\subsubsection{Cartan model for equivariant cohomology} The Cartan model for equivariant cohomology is the differential graded algebra
$$C^*_G(M):=  ( S{\gr g}^* \otimes \Omega^\bullet M)^G$$
endowed with the differential $d+\iota$. Therefore there is a natural homomorphism of
differential graded algebras \begin{eqnarray} \label{homomorphism Cartan to De Rham}
i: (C^*_G(M), d+\iota) \to \left( C^*(G, S{\gr g}^* \otimes \Omega^\bullet M
) , d_G \right)\end{eqnarray} given by the inclusion
$$( S{\gr g}^* \otimes \Omega^\bullet M)^G \subset C^0(G,
S{\gr g}^* \otimes \Omega^\bullet M )$$ since the restriction of $d_G$ to
$( S{\gr g}^* \otimes \Omega^\bullet M)^G$ is precisely $d + \iota$ because  the
operators $\bar{d}$ and $\bar{\iota}$ act trivially on the
invariant elements of $C^0(G, S{\gr g}^* \otimes \Omega^\bullet
M)$.

The induced map on cohomologies
$$i : H^*(C^*_G(M), d+\iota) \to H^*(G, S{\gr g}^* \otimes \Omega^\bullet M
)$$ is far from being an isomorphism as the case of $M=pt$  and $G=GL(1,\real)_+=\real_+^*$ shows.\\
 In this case $C_G(M)= S ( \gr{gl}(1,\real))= \real[x]$ and $d+ \iota=0$, hence $H^*(C_G(M))=
\real[x]$ where $|x|=2$. On the other hand  $H^*(G,S{\gr g}^*)= H^*(BG,\real)=\real$ since $BG$ is contractible.

Nevertheless, when the Lie group $G$ is compact,
the map $i$ induces an isomorphism in cohomology \cite{Getzler}. Now, in order to understand in more detail the relation between the cohomology of the Cartan model and the equivariant cohomology we will introduce a spectral sequence suited for this purpose.

\subsubsection{A spectral sequence for the equivariant De Rham complex}
\label{subsubsection spectral sequence}

Let us filter the complex $C^*(G, S{\gr g}^* \otimes
\Omega^\bullet M )$ by the degree in $S{\gr g}^* \otimes
\Omega^\bullet M$; namely, if we consider maps
$$f: G^k \to S{\gr g}^* \otimes
\Omega^\bullet M$$ with image homogeneous elements of degree $l$,
we will denote $deg_1(f)=k$ and $deg_2(f) = l$. Then we can define
the filtration
$$F^p:= \{ f \in C^*(G, S{\gr g}^* \otimes
\Omega^\bullet M | deg_2(f)\geq p \}$$ where  $F^{p+1} \subset
F^p$. We have that the differentials have the following degrees:
\begin{align*}
deg_1(d)& =0 &  deg_2(d)& =1\\
deg_1(\iota)&=0  & deg_2(\iota)& =1\\
deg_1(\overline{d})&=1 & deg_2(\overline{d})&=0\\
deg_1(\overline{\iota})&=-1& deg_2(\overline{\iota})&=2\end{align*}
and therefore the filtration is compatible with the differentials.

The spectral sequence associated to the filtration $F^*$ has for
page 0:
$$E_0 = \bigoplus_p F^p/F^{p+1} \cong C^*(G, S{\gr g}^* \otimes
\Omega^\bullet M )$$ and the 0-th differential is $d_0 =
\overline{d}$ because the other three differentials raise $deg_2$.
Therefore the page 1 is:
$$E_1 = H^*(C^*(G, S{\gr g}^* \otimes
\Omega^\bullet M ), \overline{d})$$ the differentiable cohomology of
$G$ with coefficients in the representation $S{\gr g}^* \otimes
\Omega^\bullet M$. The 0-th row of the first page is precisely the Cartan complex
$$E_1^{*,0} = H^0(C^*(G, S{\gr g}^* \otimes
\Omega^\bullet M ), \overline{d}) = (S{\gr g}^* \otimes
\Omega^\bullet M)^G =C_G(M)$$
and the first differential on this row  $d_1 : E_1^{*,0} \to E_1^{*+1,0}$ becomes precisely the Cartan differential $d + \iota$. Therefore we see that on the second page we get that
$$E_2^{*,0} \cong H^*(C_G(M), d+ \iota),$$
namely that the 0-th row of the second page is isomorphic to the cohomology of the Cartan complex.

Therefore we conclude that the composition
$$E_2^{*,0} \to H^*(G, S{\gr g}^* \otimes \Omega^\bullet M)$$
of the surjective homomorphism $E_2^{*,0}  \to E_\infty^{*,0} $ with the inclusion
$$E_\infty^{*,0} \subset H^*(G, S{\gr g}^* \otimes \Omega^\bullet M)$$
is equivalent to the induced map on cohomologies
$$i : H^*(C^*_G(M), d+\iota) \to H^*(G, S{\gr g}^* \otimes \Omega^\bullet M
)$$
defined previously.

\subsection{Differentiable cohomology of Lie groups}

Notice that in De Rham model for equivariant cohomology the operator $\overline{d}$ is defined in a similar fashion
as the differential for group cohomology in the case of discrete groups. The cohomology groups defined by the differential $\overline{d}$
are called the {\it differentiable cohomology groups} and are defined for any $G$-module $V$; i.e. if $G$ is a Lie group and $V$ is a $G$-module then the {\it differentiable cohomology of $G$ with values
in $V$} is the cohomology of the complex $C_d^*(G,V)$ where $C_d^k(G,V)$ consists of differentiable maps $f: G^k \to V$ such
that $f$  vanishes if any of the arguments $g_i$ equals the identity of
$G$, and the differential is $\overline{d}$ is defined by
\begin{eqnarray*}
(\bar{d}f)(g_0, \dots , g_k) & = & f( g_1, \dots , g_k ) +
 \sum_{i=1}^k (-1)^i f(g_0, \dots, g_{i-1}g_i, \dots  , g_k )\\
 & & +(-1)^{k+1} g_k f(g_0, \dots , g_{k-1} ).
\end{eqnarray*}
We denote this cohomology by
$$H_d^*(G,V) := H^*( C^*(G,V), \overline{d}).$$

For $V$ in the category of topological $G$-modules, the cohomology groups $H_d^*(G,V)$ can be seen as the relative derived functor associated
to the $G$ invariant submodule $V^G$. In particular we have that
$$H_d^0(G,V)=V^G,$$
and whenever $G$ is compact the functor of taking the $G$-invariant submodule is exact and therefore
in that case $H_d^{*>0}(G,V)=0$.

\begin{rem}
The first page of the spectral sequence that was defined in section \S \ref{subsubsection spectral sequence}
$$E_1 = H^*(C^*(G, S{\gr g}^* \otimes
\Omega^\bullet M ), \overline{d})$$
is precisely the differential cohomology defined above
$$E_1= H_d^*(G, S{\gr g}^* \otimes
\Omega^\bullet M )$$
for the $G$-module $S{\gr g}^* \otimes
\Omega^\bullet M$.
\end{rem}

For $V$ a vector space over $\real$, by differentiating the functions from $G^k$ to $V$, Van Est \cite{VanEst1} proved that for $G$ connected
$$H_d^*(G;V) \cong H^*(\gr{g},\gr{k};V)$$
whenever $K$ is the maximal compact subgroup of $G$, $\gr{k}$ and $\gr{g}$ are their corresponding Lie algebras, and
$H^*(\gr{g},\gr{k};V)$ denotes the Lie algebra cohomology defined by Chevalley and Eilenberg in \cite{ChevalleyEilenberg}.

Whenever the group $G$ has a compact form $G_u$, i.e. a compact Lie group whose complexification is isomorphic to the complexification
of $G$
$$G_\complex \cong (G_u)_\complex,$$
then
we have that
$$H^*(\gr{g},\gr{k}; V) \cong H^*(\gr{g}_u,\gr{k}; V) \cong H^*(G_u/K;V),$$
where the second isomorphism was proved by Chevalley and Eilenberg in \cite{ChevalleyEilenberg} for compact Lie groups,
and the first isomorphism follows from the isomorphisms
$$H^*(\gr{g},\gr{k}; V) \otimes_\real \complex \cong H^*(\gr{g}_\complex,\gr{k}_\complex; V_\complex) \cong
H^*(\gr{g}_u,\gr{k}; V) \otimes_\real \complex ,$$ which follow from
the isomorphism of complex Lie algebras $\gr{g}_\complex \cong
(\gr{g}_u)_\complex$. In this case we have that the differentiable
cohomology of $G$ can be calculated by topological methods, i.e.
$$H_d^*(G;V) \cong H^*(G_u/K; V).$$

\subsubsection{Example $G=SL(n,\real)$} \label{section differentiable cohomology of SL(n,R)} Let us
consider the non-compact group $G=SL(n,\real)$.
In this case we have \begin{align*}
\gr{g}&=\gr{sl}(n,\real) &
\gr{g}_\complex&= \gr{sl}(n,\complex)\\
\gr{g}_u&= \gr{su}(n) &
\gr{k}&= \gr{so}(n)\\
\gr{k}_\complex&=\gr{so}(n,\complex) &
K&=SO(n)\\
G_u&=SU(n). & &\end{align*}
Therefore for $n>2$ we have that
$$H_d^*(SL(n,\real),\real) = H^*(SU(n)/SO(n); \real) =\Lambda[h_3,h_5,...,h_{\langle n \rangle}],$$
where the degree of $h_i$ is $2i-1$ and $\langle n \rangle$ is the largest odd integer which is less or equal than $n$.

The equivariant De Rham cohomology of this group is $H^*(G,S\gr{g}^*)$. The spectral sequence defined in section
\S \ref{subsubsection spectral sequence} has  for first page
$$E_1 = H_d^*(G, S\gr{g}^*) \cong H^*(\gr{g},\gr{k}; S\gr{g}^*)$$
and since the algebra $\gr{g}=\gr{sl}(n,\real)$ is reductive and $S\gr{g}^*$ is a finite dimensional semi-simple $\gr{g}$-module in each degree \cite{Berger-et-al}, then we have that
$$H^*(\gr{g},\gr{k}; S\gr{g}^*) \cong H^*(\gr{g},\gr{k}; \real) \otimes (S\gr{g}^*)^{G}$$
and therefore
$$E_1 \cong H_d(G;\real) \otimes (S\gr{g}^*)^G.$$

The ideal of $G$-invariant polynomials  is known to be
$$(S\gr{g}^*)^G \cong \real[c_2,c_3,...,c_n],$$
where the degree of $c_i$ is $2i$, and so we get that the first page of the spectral sequence
converging to $H^*(G,S \gr{g}^*)$ is
$$E_1 \cong \Lambda[h_3,h_5,...,h_{\langle n \rangle}] \otimes \real[c_2,c_3,...,c_n].$$

Since we know that  for $n>2$
$$H^*(G,S\gr{g}^*)= H^*(BG;\real) = H^*(BSO(n);\real) \cong \real[c_2,c_4,...c_{2[n/2]}] $$
is the free algebra on the Pontrjagin classes, we obtain that the
$(2i)$-th differential of the spectral sequence maps the class $h_i$
to the class $c_i$ $$h_i \stackrel{d_{2i}}{\mapsto} c_i.$$

In particular we obtain that
\begin{equation} \label{fourth cohomology of BSL(n,R)}
\real=H^4(G,S\gr{g}^*) \cong E_\infty^{4,0} \cong E_1^{4,0} = (S^2 \gr{g}^*)^G,\end{equation}
and therefore we see that the fourth cohomology group of $BG$ is generated by
the $G$-invariant quadratic forms in  $(S^2 \gr{g}^*)^G$.

\section{Equivariant extension of differential forms} \label{section extension differential form}

In many instances in geometry, the action of a compact Lie group on a manifold being of a certain kind is equivalent to the existence of an equivariant lift
 of a specific invariant closed differential form
on the Cartan model. Some examples of this phenomenon are the following:
\begin{itemize}
\item The action of $G$ on a symplectic manifold $(M,\omega)$ being Hamiltonian is equivalent to the existence of
a closed equivariant lift $\tilde{\omega}= \omega + \mu_a\Omega^a \in C_G^2(M)$ of the symplectic form; in this case the maps $\mu_a: M \to \real$
may be assembled into a map $\mu: \gr{g} \to C^\infty M$, $\mu(a):=\mu_a$, which becomes the moment map (see \cite{AtiyahBott}).
\item The action of $G$ on an exact Courant algebroid $(TM \oplus T^*M, [,]_H)$, with Courant-Dorfman bracket twisted by the closed three form $H$, is Hamiltonian provided there exists a closed equivariant lift $\tilde{H}= \omega + \xi_a\Omega^a \in C_G^3(M)$ (see \cite{Bursztyn, CaviedesHuUribe, Hu2, Uribe-dg}).
\item Let $\Gamma$ be a connected, simple, simply connected and compact matrix group, i.e.  $\Gamma \subset GL(N,\real)$, and denote by $\Gamma_L \times \Gamma_R$ the product
of two copies of $\Gamma$ acting on $\Gamma$ on the left by the action
\begin{align*}
(\Gamma_L \times \Gamma_R) \times \Gamma & \to \Gamma \\
((g,h),k) & \mapsto gkh^{-1}.
\end{align*}
A subgroup $G \subset \Gamma_L \times \Gamma_R$ is called an {\it
anomaly-free subgroup} if there the WZW action could be gauged with
respect to the group action given by $G$. In \cite{Witten} Witten
showed that the anomaly-free subgroups are precisely the subgroups
$G$ on which the WZW term
$$\omega = \frac{1}{12 \pi} {\rm{Tr}}(g^{-1} dg)^3$$
could be lifted to a closed equivariant 3-form in the Cartan complex $\tilde{\omega}= \omega - \lambda_a \Omega^a \in C^3_G(\Gamma)$.
(We will elaborate this construction in the next chapter).
\end{itemize}

The previous examples are not exhaustive, but they give the idea of
the general principle. In these cases we have a $G$-invariant closed
form $\omega$ and we need to find a closed equivariant lift
$\tilde{\omega}$. Note that since the group $G$ is compact, the
existence of a closed equivariant lift in the Cartan model is
equivalent to the existence of a lift of the cohomology class
$[\omega] \in H^*(M; \real)$ on the equivariant cohomology $H^*(EG
\times_G M; \real)$. Therefore the obstructions of the existence of
the equivariant lift could be studied via several methods, for
instance, with the use of the Serre spectral sequence associated to
the fibration $M \to EG \times_G M \to BG$, or with the use of the
spectral sequence associated to a filtration of the Cartan complex
$C_G(M)$ given by the degree of $S \gr{g}^*$ as it is done in the papers
\cite{Figueroa, Wu}.

But what happens in the case that the Lie group $G$ is not compact?
We speculate that the situation should be similar, in the sense that
the action being of certain kind is equivalent to the existence of a
lift on the equivariant De Rham complex defined in section
\S\ref{equivariant De Rham complex}. This situation has not been
explored so far but we believe that the equivariant De Rham complex
provides a framework in which actions on non-compact Lie groups
could be better understood.

Since in many instances the geometric information is captured by
closed forms in the Cartan complex, we would like to study the
relation between the cohomology of the Cartan complex and the
equivariant De Rham cohomology whenever the group is not compact.

\subsection{Cartan complex vs. equivariant De Rham complex}

Let us consider the diagram of complexes
$$\xymatrix{
C^*_G(M) \ar[d] \ar[r]^(.35){i} & C^*(G, S{\gr g}^* \otimes \Omega^\bullet M
) \ar[ld] \\
\Omega^\bullet M & }$$
where the horizontal  map is the injective map of complexes defined in \eqref{homomorphism Cartan to De Rham}
and the vertical maps are the natural forgetful maps. Let us take a closed $G$-invariant form $H$ on $M$ and let us suppose that
we can lift this closed form to a closed form in the Cartan model $\overline{H} \in C_G(M)$, then equivariant form $i\overline{H}$
becomes a closed lift for $H$ in the complex of equivariant De Rham forms, i.e. $i\overline{H} \in
 C^*(G, S{\gr g}^* \otimes \Omega^\bullet M)$. We have then

 \begin{lemma} \label{extension in Cartan implies equivariant extension}
 Take a closed $G$-invariant form $H \in (\Omega^\bullet M)^G$. If $H$ can be lifted to a closed form $\overline{H}$ in the Cartan model of the equivariant
 cohomology, then the form $i(\overline{H})$ is a closed lift for $H$ in $ C^*(G, S{\gr g}^* \otimes \Omega^\bullet M)$, the closed forms of the
 equivariant De Rahm complex.
 \end{lemma}

Note in particular that Lemma \ref{extension in Cartan implies equivariant extension} implies that if one can extend an invariant closed form
to a closed form in the Cartan model for equivariant cohomology, then the cohomology class $[H]$ lies in the image of the canonical
forgetful homomorphism
$$H^*(G, S{\gr g}^* \otimes \Omega^\bullet M) \to H^*(M),$$
and therefore the cohomology class $[H]$ could be extended to an equivariant cohomology class in any model for the equivariant cohomology of $M$.

The converse of Lemma  \ref{extension in Cartan implies equivariant extension}  would say
that if one knows that an invariant differentiable form $H$ could be lifted to a closed form in $ C^*(G, S{\gr g}^* \otimes \Omega^\bullet M)$
then a lift could be written as a closed element in the Cartan model. The converse of Lemma  \ref{extension in Cartan implies equivariant extension}  is indeed true whenever the Lie group $G$ is compact, but
for general group actions it does not hold.

From the spectral sequence defined in section \S \ref{subsubsection spectral sequence} we have seen that in the first page we have
$$E_1^{*,0} \cong (S \gr{g}^* \otimes \Omega ^\bullet M)^G = C_G(M)$$
with differential $d_1 : E_1^{*,0} \to E_1^{*+1,0}$ equivalent to $d + \iota$. Therefore on the second page we get that
\[E_2^{*,0} \cong H^*(C_G(M), d+ \iota),\]
namely that the 0-th row of the second page is isomorphic to the cohomology of the Cartan complex.

 Since we have the surjective homomorphism $E_2^{*,0}  \to E_\infty^{*,0} $ we can conclude that

 \begin{proposition} \label{conditions extension in Cartan}
 For a closed $G$-invariant form $H \in (\Omega^\bullet M)^G$, it can be lifted to a closed form in the Cartan complex if,
 firstly the cohomology class $[H]$ could be lifted to an equivariant cohomology class in $H^*(G, S{\gr g}^* \otimes \Omega^\bullet M)$, and secondly, if the lift lies on the subgroup $E_\infty^{*,0} \subset H^*(G, S{\gr g}^* \otimes \Omega^\bullet M)$.
 \end{proposition}

 The second condition of Proposition \ref{conditions extension in Cartan} is more difficult to check than the first one, since it depends
 explicitly on the equivariant De Rham model for equivariant cohomology; for the first condition any model for the equivariant cohomology
 works.
 In certain specific situations, extensions in the Cartan model of closed forms may be obtained, and this is the subject of the next and final chapter. 
 
 We note here that a sequence of obstructions for lifting a $G$-invariant form to a closed form in the Cartan complex can 
 be determined, when studying the spectral sequence associated to appropriate filtrations of the Cartan complex, as it is carried out in
 \cite{Figueroa, Wu}. Our approach is different since we are interested in using the fact that an extension in the Cartan model can only
 exist if there is an equivariant extension, i.e. an extension in the homotopy quotient.

 \section{Equivariant extensions of the WZW term for $SL(n,\real)$ actions}

 In the physics literature (see \cite{Witten} and the references therein) it has been argued that the
 condition of anomaly cancelation for the gauged WZW action is given by the equation
 \begin{align*}
 {\rm{Tr}}(T_{a,L} T_{b,L}- T_{a,R} T_{b,R})=0
 \end{align*}
and this equation is moreover equivalent to the existence of an equivariant extension of the WZW term
 $$\omega = \frac{1}{12 \pi} {\rm{Tr}}(g^{-1} dg)^3$$
 on the Cartan model for equivariant cohomology. In the case that the group that we are gauging is compact, the anomaly cancellation is  equivalent to the existence of an equivariant lift of the cohomology class $[\omega]$, and therefore the anomaly cancellation becomes topological and could be checked with topological methods.

 In this section we study in detail the case on which the gauge group is $G=SL(n,\real)$ (or any subgroup of it)
 and we show that the anomaly cancelation condition for the WZW action is also topological and only depends on the existence
 of an equivariant lift of the cohomology class $[\omega]$; in this way we find a large family of $SL(n,\real)$ actions with anomaly cancellation.

 Let us start by recalling the explanation given by Witten \cite{Witten} that asserts that the
 condition for anomaly cancelation is equivalent to the existence of a lift in the Cartan complex of the form
$\omega$.

 \subsection{Gauged WZW actions} \label{section gauged WZW actions}

  Let $\Gamma$ be a connected and simple matrix group, i.e.  $\Gamma \subset GL(N,\real)$, such that its fundamental group
  is finite.
  Denote by $\Gamma_L \times \Gamma_R$ the product
of two copies of $\Gamma$ acting on $\Gamma$ on the left by the action
\begin{align*}
(\Gamma_L \times \Gamma_R) \times \Gamma & \to \Gamma \\
((g,h),k) & \mapsto gkh^{-1}
\end{align*}
and consider a subgroup $G \subset \Gamma_L \times \Gamma_R$  acting on the left on $\Gamma$ by the induced action of
$\Gamma_L \times \Gamma_R$.

The embedding $G \subset \Gamma_L \times \Gamma_R$ determines a map at the level of Lie algebras which can be written
as $$a \mapsto (T_{a,L}, T_{a,R}), \ \ a \in \gr{g}$$
and the canonical vector fields $X_a$ on $\Gamma$ generated by the $G$ action could be written as
$$(X_a)_g = T_{a,L} g-g T_{a,R}$$
  for all $g \in \Gamma$.

  The matrix 1-forms $g^{-1}dg$ and $dg  g^{-1}$ satisfy the equations
  \begin{align*}
  \iota_{X_a} (g^{-1}dg) & = g^{-1} T_{a,L} g - T_{a,R},\\
   \iota_{X_a}(dg g^{-1}) & =   T_{a,L} - g^{-1} T_{a,R} g,\\
  d(dg g^{-1})^{2p+1} &=  -(dg g^{-1})^{2p+2},\\
  d(g^{-1}dg)^{2p+1} &=  ( g^{-1}dg)^{2p+2}
  \end{align*}
and we can take the differential form
$$\omega = \frac{1}{12 \pi} {\rm{Tr}}(g^{-1} dg)^3,$$
which defines the WZW action. The form $\omega \in \Omega^3 \Gamma$ is $\Gamma_L \times \Gamma_R$ invariant, therefore
$\omega \in (\Omega^3 \Gamma)^G$,  it is closed and is a generator of the cohomology group $H^3(\Gamma)=\real$.

To find a closed extension of $\omega$ in the Cartan complex we need
to find 1-forms $\lambda_a$ such that the following equations are
satisfied:
\begin{align*}
\iota_{X_a} \omega - d \lambda_a=0, \\
\iota_{X_a} \lambda_b + \iota_{X_b} \lambda_a=0,\\
\LL_{X_b} \lambda_a = \lambda_{[b,a]},
\end{align*}
where the first two imply that the form $\tilde{\omega}=\omega - \lambda_a \Omega^a$ is $(d+\iota)$-closed, and the third one
implies that $\tilde{\omega}$ is $G$-invariant. Calculating $\iota_{X_a} \omega$ we obtain
\begin{align*}
\iota_{X_a} {\rm{Tr}}(g^{-1} dg)^3 &= 3 {\rm{Tr}}\left((g^{-1} T_{a,L} g - T_{a,R}) (g^{-1}dg)^2 \right)\\
& = 3{\rm{Tr}}\left( T_{a,L} (dg g^{-1})^2 - T_{a,R} (g^{-1}dg)^2 \right)\\
&= d \left[3{\rm{Tr}}\left( T_{a,L} (dg g^{-1}) +T_{a,R} (g^{-1}dg) \right) \right]
\end{align*}
 and therefore we see that we can define
 $$\lambda_a= \frac{1}{4 \pi} {\rm{Tr}}\left( T_{a,L} (dg g^{-1}) +T_{a,R} (g^{-1}dg) \right)$$
 satisfying the equation $\iota_{X_a} \omega - d \lambda_a=0$.
  The fact that $\LL_{X_b} \lambda_a = \lambda_{[b,a]}$ is satisfied, is a tedious but straightforward computation.

 Now we compute
 \begin{align*}
\iota_{X_a} \lambda_b + \iota_{X_b} \lambda_a &= \frac{1}{2 \pi} {\rm{Tr}} \left( T_{a,L}T_{b,L} - T_{a,R}T_{b,R} \right)
 \end{align*}
noting that for $a,b \in \gr{g}$ the function thus defined become constant, and therefore we have that
 $$\tilde{\omega}=\omega - \lambda_a \Omega^a \in (S \gr{g}^* \otimes \Omega^\bullet \Gamma)^G$$
 is a 3-form in the Cartan complex $C_G(\Gamma)$ and its failure to be closed is the quadratic form
 \begin{align*}
 (d + \iota) \tilde{\omega} \in (S^2 \gr{g}^*)^G,
 \end{align*}
 where the coefficient of $\Omega^a \Omega^b$ of this quadratic form is precisely
 $$-\frac{1}{2 \pi} {\rm{Tr}} \left( T_{a,L}T_{b,L} - T_{a,R}T_{b,R} \right).$$

Since it was known in the literature that the condition for the absence of anomalies was
$${\rm{Tr}} \left( T_{a,L}T_{b,L} - T_{a,R}T_{b,R} \right)=0,$$
Witten concluded that the absence of anomalies was equivalent to the existence of a closed extension of $\omega$ in the Cartan complex.

Whenever the Lie group $G$ is compact, the existence of such an extension is equivalent to the existence of an equivariant
extension of the cohomology class $[\omega]$, and this can be checked with the use of the Serre spectral sequence associated
to the fibration
$$\Gamma \to EG \times_G \Gamma \to BG.$$
The second page of the spectral sequence becomes
$$E_2^{p,q}= H^p(BG; H^q(\Gamma ;\real)) \cong H^p(BG;\real) \otimes H^q(\Gamma ;\real),$$
and since $H^1(\Gamma ;\real)=H^2(\Gamma ;\real)=0$, the only
non-trivial differential that affects $[\omega] \in E_2^{0,3}$ is
$d_4$ thus defining an element
$$d_4([\omega]) \in E_2^{4,0} \cong H^4(BG;\real);$$
this implies that the only obstruction to lift $[\omega]$ to an equivariant class is precisely $d_4([\omega])$.
Since we assumed that $G$ is compact, we know that
$$H^4(BG;\real) \cong (S^2 \gr{g}^*)^G$$
and therefore we must have that
$$d_4([\omega]) = (d+\iota)\tilde{\omega},$$
namely that the two obstructions are the same.

The previous argument permits to find several cases on which there is anomaly cancellation. The simplest of all is the adjoint
action of $G$ on itself $\Gamma=G$ since in this case the spectral sequence associated to the fibration $EG \times_G G_{ad} \to BG$
always collapses at the second page, and therefore $d_4=0$.

 \subsection{WZW actions with gauge group $G=SL(n,\real)$}
 In this section we will argue that if the gauge group is $G=SL(n,\real)$ then the cancellation of anomalies is topological, and therefore
 it is equivalent to  the existence of an equivariant lift of the cohomology class $[\omega]$.

 \begin{theorem} \label{equivalence of extension for SL(n,R)}
 Let $\Gamma$ be a connected, simple matrix group with finite fundamental group.
 Let $G=SL(n,\real)$ and consider an action on $\Gamma$ defined by an injection $G \subset \Gamma_L \times \Gamma_R$.
 The existence of an equivariant extension on $H^3(EG \times_G \Gamma; \real)$ of the cohomology class $[\omega]$ is equivalent
 to the existence of a closed lift to the Cartan model of $\omega$.
  \end{theorem}

\begin{proof}
We know that if there is an extension $\tilde{\omega}$ in the Cartan model, then the cohomology class $[\tilde{\omega}]$ represents
the lift in the cohomology group $H^3(G, S \gr{g}^* \otimes \Omega^\bullet \Gamma)$. To prove the converse we will make use of the constructions and results of sections  \S \ref{section differentiable cohomology of SL(n,R)}, \S \ref{section extension differential form} and \S\ref{section gauged WZW actions}.

From the equivariant De Rham theorem we know that $H^*(BG;\real) \cong H(G, S \gr{g}^*)$ and therefore we could
take the class $d_4([\omega])$, which is the obstruction of extending $[\omega]$ to an equivariant class, to be an element in
$$d_4([\omega]) \in H^4(G, S \gr{g}^*).$$

We already know that $(d+ \iota) \tilde{\omega} \in (S^2 \gr{g}^*)^G$ and its cohomology class in $H^4(G, S \gr{g}^*)$ represents
the same obstruction for an equivariant lift, i.e.
$$d_4([\omega]) = [(d+ \iota) \tilde{\omega}].$$
In general it may happen that the cohomology class $ [(d+ \iota) \tilde{\omega}]$ is zero even though the form $(d+ \iota) \tilde{\omega}$
is different from zero. But in the particular case of $G=SL(n, \real)$ we have already seen in \eqref{fourth cohomology of BSL(n,R)}
that the inclusion map of the Cartan complex into the equivariant De Rham complex
$$(S\gr{g}^*)^G \to C^*(G, S\gr{g}^*)$$
induces an isomorphism in degree 4
$$(S^2\gr{g}^*)^G \stackrel{\cong}{\to} H^4(G,S\gr{g}^*), \ \ (d+\iota)\tilde{\omega} \mapsto [(d+\iota)\tilde{\omega}]$$
and therefore we have that the vanishing of the class $d_4([\omega])$ is equivalent to the vanishing of the quadratic form $(d+\iota)\tilde{\omega}$, i.e.
$$d_4([\omega])=0 \ \ \mbox{if and only if} \ \ (d+\iota)\tilde{\omega}=0.$$
\end{proof}

We see that for the case on which the gauge group is $SL(n,\real)$, the equations of cancellation of anomalies, namely that
for all $a,b \in \gr{sl}(n,\real)$
$${\rm{Tr}} \left( T_{a,L}T_{b,L} - T_{a,R}T_{b,R} \right)=0,$$
are equivalent to the existence of an equivariant extension on $H^3(EG \times_G \Gamma; \real)$ of the cohomology class
$[\omega]$. Now we are ready to give examples on both the existence and the non existence of equivariant extensions of $\omega$.

\subsection{Examples}
\subsubsection{Adjoint action}

Let $G=\Gamma$ and consider the diagonal injection
$$G  \subset \Gamma_L \times \Gamma_R, \ \ g \mapsto (g,g)$$
which induces the adjoint action of $G$ on $\Gamma=G^{\rm{ad}}$. In this case the cohomology of the homotopy quotient  $EG \times_G G^{\rm{ad}}$ is isomorphic to the cohomology of $G$ tensor the cohomology of $BG$:
$$H^*(EG \times_G G^{ad}; \real)\cong   H^*(BG;\real) \otimes H^*(G;\real).$$
This isomorphism can be obtained from the Serre spectral sequence associated to the fibration
$$G \to EG \times_G G^{\rm{ad}} \to BG,$$
which collapses at level 2 because the classes in $H^*(G;\real)$
can be lifted to classes to $H^*(EG \times_G G^{ad}; \real)$: take
a primitive class in $H^*(BG;\real)$ (namely a class in
$H^*(BG;\real)$ which is in the image of a primitive element in
$H^*(G)$ of one of the differentials of the Serre Spectral
Sequence associated to the fibration $G \to EG \to BG$) pull it
back to $S^1 \times \LL BG$ via the evaluation map where $\LL BG$
is the space of free loops of $BG$, then integrate over $S^1$ and
get a class in $H^*(\LL BG;\real)$ of degree 1 less; this class in
$H^*(\LL BG;\real)$ once restricted to the based loops $\Omega BG
\simeq G$ is precisely the class in $H^*(G)$ that defined the
primitive class in $H^*(BG)$ that we started with; finally recall
that $\LL BG \simeq EG \times_G G^{ad}$.

If we take $G=\Gamma=SL(n,\real)$, we know by Theorem
\ref{equivalence of extension for SL(n,R)} that the existence of an
equivariant extension of $[\omega]$ is equivalent to the cancelation
of anomalies for this gauged action, and since the class $[\omega]$
can be extended to an equivariant one, we conclude that in this case
there is an anomaly cancelation.

In particular, for any subgroup $F \subset SL(n,\real)$ acting by
the adjoint action on $SL(n,\real)$ there is also cancelation of
anomalies.

\subsubsection{$G=SL(n,\real) \subset \Gamma_L$ for $n>2$ }

Whenever the action of $G=SL(n,\real)$ on $\Gamma$ is obtained by a left action induced by an inclusion $G=SL(n,\real) \subset \Gamma_L$, we have that the $G$ action on $\Gamma$ is free and therefore the homotopy quotient $EG \times_G \Gamma$ and the quotient $G \backslash \Gamma$ are homotopy equivalent. For $n>2$ the Serre spectral sequence tells us that $$d_4([\omega])=c_2 \in H^4(BSL(n,\real);\real)$$
and therefore the class $[\omega]$ does not extend to an equivariant one. By Theorem \ref{equivalence of extension for SL(n,R)} we know that this implies that there is no cancellation of anomalies. We conclude that for left free actions of the group $SL(n,\real)$ for $n>2$  there must exist $a,b \in \gr{sl}(n,\real)$ such that
$${\rm{Tr}} \left( T_{a,L}T_{b,L} - T_{a,R}T_{b,R} \right)\neq 0.$$

\subsubsection{$G=SL(2,\real) \subset \Gamma_L$ }

Following the same argument as before, we have that
$d_4([\omega])=0$ since $H^4(BSL(2,\real),\real)=0$. Therefore for
left free actions of the group $SL(2,\real)$ there is anomaly
cancelation.

\def\cprime{$'$}


\begin{thebibliography}{10}

\bibitem{AtiyahBott}
M.~F. Atiyah and R.~Bott.
\newblock The moment map and equivariant cohomology.
\newblock {\em Topology}, 23(1):1--28, 1984.

\bibitem{Berger-et-al}
M.~Berger et~al.
\newblock {\em S\'eminaire ``{S}ophus {L}ie'' de l'{E}cole {N}ormale
  {S}up\'erieure, 1954/1955. {T}h\'eorie des alg\`ebres de {L}ie. {T}opologie
  des groupes de {L}ie}.
\newblock Secr\'etariat math\'ematique, 11 rue Pierre Curie, Paris, 1955.

\bibitem{BottShulmanStasheff}
R.~Bott, H.~Shulman, and J.~Stasheff.
\newblock On the de {R}ham theory of certain classifying spaces.
\newblock {\em Adv. Math.}, 20(1):43--56, 1976.

\bibitem{Bursztyn}
H.~Bursztyn, G.~R.~Cavalcanti and M.~Gualtieri.
\newblock Reduction of {C}ourant algebroids and generalized complex structures.
\newblock {\em Adv. Math.}, 211(2):726--765, 2007.

\bibitem{Cartan}
H.~Cartan.
\newblock La transgression dans un groupe de {L}ie et dans un espace fibr\'e
  principal.
\newblock In {\em Colloque de topologie (espaces fibr\'es), {B}ruxelles, 1950},
  pages 57--71. Georges Thone, Li\`ege, 1951.

\bibitem{CaviedesHuUribe}
A.~Caviedes, S.~Hu and B.~Uribe.
\newblock Chern-{W}eil homomorphism in twisted equivariant cohomology.
\newblock {\em Differential Geom. Appl.}, 28(1):65--80, 2010.

\bibitem{ChevalleyEilenberg}
C.~Chevalley and S.~Eilenberg.
\newblock Cohomology theory of {L}ie groups and {L}ie algebras.
\newblock {\em Trans. Amer. Math. Soc.}, 63:85--124, 1948.

\bibitem{Figueroa}
J.~M. Figueroa-O'Farrill and S.~Stanciu.
\newblock Gauged {W}ess-{Z}umino terms and equivariant cohomology.
\newblock {\em Phys. Lett. B}, 341(2):153--159, 1994.

\bibitem{GarciaCompeanPaniagua}
H.~Garc{\'{\i}}a-Compe{\'a}n and P.~Paniagua.
\newblock Gauged {WZW} models via equivariant cohomology.
\newblock {\em Mod. Phys. Lett. A}, 26(18):1343--1352, 2011.

\bibitem{Getzler}
E.~Getzler.
\newblock The equivariant {C}hern character for non-compact {L}ie groups.
\newblock {\em Adv. Math.}, 109(1):88--107, 1994.

\bibitem{Guillemin}
V.~W.~Guillemin and S.~Sternberg.
\newblock {\em Supersymmetry and equivariant de {R}ham theory}.
\newblock Mathematics Past and Present. Springer-Verlag, Berlin, 1999.


\bibitem{Hu2}
S.~Hu.
\newblock Reduction and duality in generalized geometry.
\newblock {\em J. Symplectic Geom.}, 5(4):439--473, 2007.

\bibitem{Paniagua}
P.~Paniagua.
\newblock {\em Anomal{\'\i}as, Teor{\'\i}a de Wess-Zumino-Witten y
  cohomolog{\'\i}a equivariante}.
\newblock PhD thesis, CINVESTAV, M\'exico, 2010.

\bibitem{Uribe-dg}
B.~Uribe.
\newblock Group actions on dg-manifolds and exact courant algebroids.
\newblock {\em Comm. Math. Phys.}, 318(1):35--67, 2013.

\bibitem{VanEst1}
W.~T. Van Est.
\newblock Group cohomology and {L}ie algebra cohomology in {L}ie groups. {I},
  {II}.
\newblock {\em Nederl. Akad. Wetensch. Proc. Ser. A. {\bf 56} = Indagationes
  Math.}, 15:484--492, 493--504, 1953.

\bibitem{Witten}
E.~Witten.
\newblock On holomorphic factorization of {WZW} and coset models.
\newblock {\em Comm. Math. Phys.}, 144(1):189--212, 1992.

\bibitem{Wu}
S.~Wu. 
\newblock Cohomological obstructions to the equivariant extension of closed invariant forms.
\newblock {\em J. Geom. Phys.},
 10(4):381--392, 1993.

\end{thebibliography}

\end{document}